\documentclass[12pt,letterpaper]{article}
\usepackage[latin1]{inputenc}
\usepackage[english]{babel}
\usepackage{amssymb,amsmath,amsthm}
\usepackage{a4wide}
\newcommand\R{\mathbb{R}}

 \newtheorem{theo}{Theorem}[section]

\begin{document}
\title{Isochronicity conditions for some planar polynomial systems}
\author{Islam Boussaada \footnote{LMRS, UMR 6085, Universite de Rouen,
Avenue de l'universit\'e, BP.12,
76801 Saint Etienne du Rouvray, France, islam.boussaada@etu.univ-rouen.fr}, A. Raouf Chouikha 
\footnote{   	
 	Laboratoire Analyse Géometrie et Aplications, URA CNRS 742, Institut Gallilée, Université Paris 13, 99 Avenue J.-B. Clément, 93430 Villetaneuse, France, chouikha@math.univ-paris13.fr} 
 	and Jean-Marie Strelcyn \footnote{LMRS, UMR 6085, Universite de Rouen,
Avenue de l'universit\'e, BP.12,
76801 Saint Etienne du Rouvray and Laboratoire Analyse Géometrie et Aplications, URA CNRS 742, Institut Gallilée, Université Paris 13, 99 Avenue J.-B. Clément, 93430 Villetaneuse, France, Jean-Marie.Strelcyn@univ-rouen.fr, strelcyn@math.univ-paris13.fr}}

\maketitle {}

\bigskip

\begin{abstract}
\begin{center}
\begin{minipage}{16cm}
We study the isochronicity of centers at $O\in \mathbb{R}^2$ for systems 
$\dot x=-y+A(x,y),\;\dot y=x+B(x,y)$, where $A,\;B\in \mathbb{R}[x,y]$, which can be reduced to the Liénard type equation.\ Using the so-called \newline C-algorithm we have found 27  new multiparameter isochronous centers. 
 
 \end{minipage}
\end{center}
\footnote{{\it Key Words and phrases:} \ polynomial systems, center, isochronicity, Liénard type equation,  Urabe function,  first integral, linearizability.\\
2000 Mathematics Subject Classification  \ 34C15, 34C25, 34C37}\\

\end{abstract}

\section{Introduction }
\subsection{Generalities}

Let us consider the system of real differential equations of the form
\begin{equation}\label{GEN1}
\frac{dx}{dt}=\dot x=-y+A(x,y),\qquad \frac{dy}{dt}=\dot y=x+B(x,y),
\end{equation}
where  $(x,y)$ belongs to an  open connected subset  $U\subset {\R}^2$, $A,B\in{C}^1(U,\R)$, where $A$ and $B$ as well as their first derivatives vanish at $(0,0)$.
An isolated
singular point $p\in U$ of system~\eqref{GEN1} is a {\it{center}}
 if  there exists a punctured neighborhood $V\subset U$ of $p$  such that
  every orbit of~\eqref{GEN1} lying  in $V$ is a closed orbit surrounding $p$.
  A center $p$ is {\it{isochronous}} if the period is constant for all closed orbits in some neighborhood of $p$.
  
  The simplest example is the linear isochronous center at the origin $O=(0,0)$ given by the system \begin{equation}\label{LINC} \dot x=-y,\; \dot y=x.\end{equation}
  
  The problem of caracterization of couples $(A,B)$ such that $O$ is an isochronous center (even a center) for the system \eqref{GEN1} is largely open.
  
  An overview  \cite{CS+} present the basic results concerning the problem of the isochronicity, see also \cite{ALS,C1+,C2+,ROSH}. 
  
   The hunting of isochronous centers is now a flourishing activity. By this paper we would like to contribute to it.\\
  
  The well known Poincaré Theorem  asserts that when $A$ and $B$ are real analytic, a center of~\eqref{GEN1} is isochronous if and only if in some real analytic coordinate system it take the form of the linear center~\eqref{LINC}.  Let us formulate now another theorem of the same vein (see for example \cite{ALS}, Th.13.1 and \cite{ROSH}, Th.4.2.1).
  
\noindent{\bf{Theorem A}}\,(\cite{MRT+}, Th.3.3)
{\it{Let us suppose that the origin $O$ is an isochronous center of system~\eqref{GEN1} with real analytic functions $A$ and $B$. Let $F(x,y)=x^2+y^2+o(|(x,y)|^2)$ be an real analytic first integral defined in some neighborhood of $O$. Then there exists a real analytic change of coordinates $u(x,y)=x+o(|(x,y)|),\;v(x,y)=x+o(|(x,y)|)$ bringing the system~\eqref{GEN1} to the linear system $\dot u=-v,\;\dot v=u$ and such that $F(x,y)=u^2(x,y)+v^2(x,y)$.}}

We now pass to the heart of the matter. To make this paper more accessible, we report all strictly technical remarks concerning $C$-algorithm and Gr\"obner basis to Appendix, Sec.7.

In some circumstances system \eqref{GEN1} can be reduced to {\it{the Liénard type equation}}
\begin{equation}\label{L2} 
\ddot x+f(x){\dot x}^2+g(x)=0
\end{equation} 
with $f,\;g \in C^1(J,\mathbb{R})$, where $J$ is some neighborhood of $0\in \mathbb{R}$ and $g(0)=0$. If it is so, the system~\eqref{GEN1} is called {\it{reducible}}.
To the equation~\eqref{L2} one associates equivalent, two dimentional (planar), Liénard type system
\begin{equation}\label{C_n}
\left.\begin{split} \dot x &= y \\ \dot y &= -g(x) - f(x) y^2
 \end{split}\right\} 
   \end{equation}

 For reducible systems considered in this paper, the nature of singular point $O$  for both system~\eqref{GEN1} and ~\eqref{C_n} is the same; in particular this concerns the centers and isochronous centers.
     More precisely, for the purpose of this paper we shall consider two cases where such a reduction is possible.
     \begin{itemize}
\item Case 1:\; When $-y+A(x,y)=-y\tilde{A}(x)$ and $x+B(x,y)=\tilde{B}(x)+\tilde{C}(x)y^2$ system \eqref{GEN1} can be written 
\begin{equation}\label{GL}
\left.\begin{split} \dot x &=-y\tilde{A}(x)\\ \dot y &= \tilde{B}(x)+\tilde{C}(x)y^2
 \end{split}\right\}
 \end{equation}
   By the change of coordinates $(u,v):=(x,-y\tilde{A}(x))$ we get
   
$$\left.\begin{split} \dot u &=v\\ \dot v &= -\tilde{A}(u)\tilde{B}(u)+\frac{\tilde{A}'(u)-\tilde{C}(u)}{\tilde{A}(u)}v^2
 \end{split}\right\}$$ 
 
  In this way we obtain the reduction to the system~\eqref{L2} with 
\begin{equation}\label{GLa} 
f(x)=-\frac{\tilde{A}'(x)-\tilde{C}(x)}{\tilde{A}(x)}  \ and \  g(x)=\tilde{A}(x)\tilde{B}(x).
\end{equation}
   
\item Case 2:\; When $A(x,y)=0$ and $B(x,y)=xP(y)$ where $P(0)=0$. 
   In this case system \eqref{GEN1} can be written 

$$\left.\begin{split} 
\dot x&=-y\\
 \dot y&=x(1+P(y))
     \end{split}\right\}$$
 By the change of coordinates $(u,v):=(y,x(1+P(y)))$ we get
\begin{equation*}
\left.\begin{split} 
\dot u&=v\\ \dot v&=-u(1+P(u))+v^2 \frac{P'(u)}{1+P(u)}
      \end{split}\right\}
       \end{equation*}
We obtain the system~\eqref{L2} with \begin{equation}\label{GLb} f(x)=-\frac{P'(x)}{1+P(x)}\ and \ g(x)=x(1+P(x)).
\end{equation}

\end{itemize}

In both cases the determinant of the Jacobian matrix of coordinate change does not vanish at $(0,0)$. Thus the nature of singular point $O$ is the same for system~\eqref{GEN1} and~\eqref{C_n}. 
     
   Let us return now to the Liénard type equation \eqref{L2}. Let us define the following  functions 
\begin{equation}\label{xixi}
F(x):= \int_0^xf(s) ds, \quad \phi(x):= \int_0^x e^{F(s)} ds.
\end{equation}

When $xg(x)>0$ for $x\neq0$, define the function $X$  by
\begin{equation}\label{xi}
\frac {1}{2} X(x)^2 = \int_0^x g(s) e^{2F(s)} ds
\end{equation}
and $xX(x)>0$ for $x\neq 0$.

Let us formulate now the following theorems which are the starting point of this paper.

\noindent{\bf{Theorem~B}}\,(\cite{S2+}, Th.1)
{\it{Let $f ,\, g \in C^1(J,\R)$. If $xg(x)>0$ for $x\neq0$, then the system~\eqref{C_n} has a center at the origin $O$.
When $f,\; g$ are real analytic , this condition is also necessary.

When $f ,\, g \in C^1(J,\R)$, the first integral of the system~\eqref{C_n} is given by the formula
 \begin{equation}\label{L2FI}
 I(x,\dot x)=2\int_0^x g(s) e^{2F(s)} ds+(\dot x e^{F(x)})^2
 \end{equation}}}

\noindent{\bf{Theorem C}}\,(\cite{C3+}, Th.2.1)
{\it{
Let $f$, $g$ be  real analytic functions defined in a neighborhood $J$ of $0\in \mathbb{R}$, 
and let $x g(x) > 0$ for $x \neq 0$.
Then  system~\eqref{C_n}  has an isochronous center at $O$ if and only if 
there exists an real analytic odd function $h$ which satisfies the following conditions 
 \begin{equation}
 \label{CRI}\frac {X(x)}{1+h(X(x))} = g(x) e^{F(x)},
\end{equation}
the function $\phi(x)$ satisfies 
 \begin{equation}
 \label{bb} \phi(x) = X(x) + \int_0^{X(x)} h(t) dt,
\end{equation}
 and $X(x)\phi (x) > 0$ for $x\neq 0$.

In particular, when $f$ and  $g$ are odd, $O$ is an isochronous center if and only if  $g(x) = e^{-F(x)}\phi (x)$, or equivalently  $h= 0$.
}}

The function $h$ is called {\it{Urabe function}}.
The above Theorem implies\\
{\bf{Corollary A}}\, (\cite{C3+}, Corollary 2.4)
{\it{Let $f$, $g$ be real analytic functions defined in a neighborhood of $0\in\R$, and $x g(x) > 0$ for $x \neq 0$.
The origin $O$ is isochronous center of system~\eqref{C_n} with Urabe function $h=0$ if and only if 
\begin{equation}\label{Null}
g'(x)+g(x)f(x)=1
\end{equation}
for sufficiently small $x$.}}

In the future we shall call the Urabe function of the isochronous center of reducible system \eqref{GEN1} the Urabe function of the corresponding Lienard type equation.

In \cite{C3+} the second author described how to use Theorem~C to build an algorithm (C-algorithm, see Sec.7.1 Appendix for more details) to look for isochronous centers at the origin for reducible system~\eqref{GEN1}, and apply to the case where $A$ and $B$ are polynomials of degree $3$.
This work was continued in \cite{CRC+}. 

The main results obtained in \cite{C3+} and \cite{CRC+}  are the necessary and sufficient conditions for isochronicity of the center at $O$ in term of parameters for the cubic system 

$$\left. \begin{split} \dot x &= - y +a x y + b x^2y\\
 \dot y &= x + a_1x^2 + a_3y^2 + a_4x^3 + a_6xy^2 \end{split}\right\}$$ 

The aim of this paper is  to extend these investigations for systems  with  higher order perturbations of the linear center $ \dot x=-y,\; \dot y=x$.  

Like in~\cite{C3+,CRC+}, our main tool to investigate the isochronous centers for multiparameters systems reducible to Liénard type equation is C-algorithm. Nevertheless, when searching only the isochronous centers with zero Urabe function the Corollary~A gives a much simpler method which is widely used in this paper. It consists in identifying the parameters values for which identity~\eqref{Null} is satisfied.

In all cases considered in \cite{C3+,CRC+} as well as in the present paper the Urabe function is of the form {\large{$h(X)=\frac{k_1X^{s}}{\sqrt{k_2+k_3 X^{2s}}}$}} where $s$ is an odd natural number, $k_1,k_2,k_3 \in \mathbb{R}$ and $k_2>0$.
Like in \cite{C3+,CRC+}, we ask if the Urabe function of corresponding Lienard type equation (called in the sequel also the Urabe function of the isochronous center under consideration) is always of the above form.

One of our contributions is the explicit description of simple multiparameter families of system \eqref{GEN1} with isochronous centers at the origin and with a very complicated coefficients. \\ 
Their complexity clearly indicates that we approach the end of purely enumerative study in this field.

Let us stress that using the change of variables given by a polynomial automorphism of $\R^2$ it is easy to transform a simple system of polynomial differential equation with isochronous center at the origin into a very complicated one. 
But systems thus obtained do not belong to the class of simple and natural systems studied in the present paper.
Our contribution is the explicit  description of such complicated systems in simple and natural multiparameter families of planar polynomial differential systems.

In our investigations we have used {\it{Maple}} in its version 10. 
To compute the Gr\"obner basis  (with DRL order) of the obtained systems of polynomial equations, we have used {\it{Salsa Software }} more precisely the  implementation {\it{FGb}} \cite{F+}.
 \subsection{Beyond the degree $3$}
 
 We now present the list of reducible systems for which we study the isochronous centers at the origin.
 \begin{enumerate}
\item In Section 2 we study the most general homogeneous perturbation of arbitrary degree $n\geq3$ of the linear center which belongs to the Case 1 from the Sec.1.1 :
 \begin{equation}\label{C_hn} 
 \left. \begin{array}{rl} \dot x&= - y+ ax^{n-1}y\\
 \dot y&= x + bx^{n-2}y^2 + cx^n \end{array}\right\} 
\end{equation}
Here we found 3 isochronous centers for even $n\geq4$ and 2 isochronous centers for odd $n\geq3$ which are new.
\item In Sections 3 and 4 we study the most general polynomial perturbation of degree four of the linear center which belongs to the Case 1 from the Sec.1.1 :
 \begin{equation}\label{C_4} 
 \left. \begin{split}\dot x&= - y+ a_{11}xy+ a_{21}x^2y + a_{31}x^3y\\
 \dot y&= x + b_{20}x^2 + b_{30}x^3+ b_{02}y^2 + b_{12}xy^2 + b_{22}x^2y^2+ b_{40}x^4 \end{split}\right\} 
\end{equation}
First using Corollary~A we identify all isochronous centers with zero Urabe function. Here we found 6 isochronous centers which are new.
The study of this system by C-algorithm can not be performed by our actual computer facilities.
Thus, we select for investigation two sub-families; the first one when $ a_{1,1}=b_{3,0}=0$ and the second one when $a_{1,1}= a_{2,1}=0$. Here we found 10 isochronous centers which are new.
\item In Section 5 we study the most general polynomial perturbation of degree five of the linear center which belongs to the Case 1 from the Sec.1.1 :
 \begin{equation}\label{C_5} 
 \left. \begin{split}\dot x&= - y+ a_{11}xy+ a_{21}x^2y + a_{31}x^3y+ a_{41}x^4y\\
 \dot y&= x + b_{20}x^2 + b_{30}x^3+ b_{02}y^2 + b_{12}xy^2 + b_{22}x^2y^2+ b_{32}x^3y^2+ b_{40}x^4+ b_{50}x^5 \end{split}\right\} 
\end{equation}
Using Corollary~A we identify all isochronous centers with zero Urabe function where $b_{50}=0$. Here we found 8 isochronous centers which are new.
\item In Section 6 we study the following Abel system of arbitrary degree $n\geq2$ which belongs to the Case 2 (see Sec.1.1) :
\begin{equation}\label{AbI}
\left. \begin{split}
\dot x&=-y\\ \displaystyle\dot y&=\displaystyle\sum_{k=0}^{n} a_k x{y}^k,
    \end{split}\right\} 
    \end{equation}
   where   $a_k \in \R$, for $k =0,\ldots, n$.  Here we verify that up to $n=9$ there are no other isochronous center than  the one found 
in [19].
 \end{enumerate}
 
 To sum up, we have found 24  multiparameter isochronous centers as well as three infinite families of them that correspond to the perturbations of arbitrary high degree, the whole of which are new. 
 
 Concerning the reduction to the Lienard type equations the systems \eqref{C_hn}-\eqref{C_5} come under case 1, while system \eqref{AbI} come under case 2 (see Sec. 1.1). In particular, for the systems \eqref{C_hn}-\eqref{C_5} the functions $f$ and $g$ from equation (3)  are those given by formulas \eqref{GLa}, while for the system \eqref{AbI} they are those given by formulas \eqref{GLb}.
 
 Let us stress that by Theorem~B, in all the above cases the origin $O$ is always a center (indeed, the condition $xg(x)>0$ for $x\neq0$ is satisfied for sufficiently small $|x|$.

When describing in Sec.3-6 the identified isochronous centers, all parameters intervening in the formulas are arbitrary, except that one always supposes that  the denominators are non zero.
To avoid  misprints all formulas are written exactly in the form produced by {\it{Maple}}. All fractions which appear in the formulas are irreducible. In all cases when we were able to write down first integrals and linearizing changes of variables, the explicite formulas are reported.


 \section{Homogeneous perturbations of arbitrary degree}
 
Taking into account the condition $g'(x)+f(x)g(x)=1$ from Corollary~A, one easily obtains the following Theorem
\begin{theo}\label{TU0} 
For $n\geq 2$ the system~\eqref{C_hn} has an isochronous center at the origin $O$ with zero Urabe function only in one of the following two cases
{\large{
\begin{equation}\label{U01}
 \left. \begin{array}{rl} \dot x&= -y+ax^{n-1}y \\
 \dot y&=x+ax^{n-2}y^2
\end{array}\right\} 
\end{equation}
\begin{equation}\label{U02}
 \left. \begin{array}{rl} \dot x&= y+{\frac {b}{n}}{x}^{n-1}y \\
 \dot y&=x+b{x}^{n-2}{y}^{2}-{\frac {(n-1)b}{{n}^{2}}} {x}^{n}
\end{array}\right\} 
\end{equation}}}
Moreover, for odd $n\geq3$ there are no other isochronous centers.
\end{theo} 
\begin{proof}
System~\eqref{C_hn} is reducible to system~\eqref{C_n} with 
$$f(x)={\frac {{x}^{n-2} \left( b+na-a \right) }{1-a{x}^{n-1}}}\;\; {\rm{and}}\;\;g(x)=\left( 1-a{x}^{n-1} \right)  \left( x+c{x}^{n} \right)$$

The condition $g'(x)+f(x)g(x)=1$ allows directly to the following two cases :
\begin{enumerate}
\item $\left\{ a=b,c=0\right\}$ which gives the system~\eqref{U01}.

\item {\large{$\left\{c=-{\frac {b \left( n-1 \right) }{{n}^{2}}},a={\frac {b}{n}}\right\} $}} which gives the system~\eqref{U02}.
\end{enumerate}
Applying formula~\eqref{L2FI} using {\it{Maple}}, one see that for $n=4,6,8$ the first integral of system~\eqref{U01} takes the form
  \begin{equation*}
  H_{\eqref{U01}}={\frac {\left({x}^{2}+{y}^{2}\right)}{ \left( -1+{\it a}\,{x}^{n-1} \right) ^{\frac{2}{n-1}}}}
  \end{equation*}
 Then Theorem~A suggest that the linearizing change of coordinates is
  \begin{equation}\label{LU01}
 u= {\frac {x}{\sqrt [n-1]{1-{\it a}\,{x}^{n-1}}}},\;\;\;\; v ={\frac {y}{\sqrt [n-1]{1-{\it a}\,{x}^{n-1}}}}
 \end{equation}
 Now one directly verifies that $H_{\eqref{U01}}$ is always a first integral of system~\eqref{U01} and using {\it{Maple}} one easily checks that~\eqref{LU01} is a linearizing change of coordinates.
 
 Exactly the same arguments work for the system~\eqref{U02}.
Its first integral is
  \begin{equation*}
  H_{\eqref{U02}}=\frac{{x}^{2} \left( 1+c{x}^{n-1} \right) ^{2}+{y}^{2} }{ \left( n-1 \right) ^{2}   \left( n-1+nc{x}^{n-1
} \right) ^{{\frac {2n}{n-1}}}}
  \end{equation*}
  and its linearizing change of coordinates is
  \begin{equation*}
 u= \frac{x \left( 1+c{x}^{n-1} \right)} { \left( n-1 \right) \left( 
 \left( n-1+nc{x}^{n-1} \right) ^{{\frac {n}{n-1}}} \right) }
,\;\;\; v =\frac{y} {\left( n-1 \right) \left(  \left( n-1+nc{x}^{n-1}
 \right) ^{{\frac {n}{n-1}}} \right)}
 \end{equation*}

\end{proof}

When $n\geq4$ is even the preliminary investigation of system \eqref{C_hn} performed by C-algorithm  strongly suggests that for such  $n$ there exists exactly one additional  isochronous center with non zero Urabe function. Its existence is proved in Theorem~\ref{HUN}. Unfortunately, its uniqueness is not yet proved for arbitrary even $n\geq4$. For $n=4,\,6,\,8$ the uniqueness was proved using {\it{Maple}} and Gr\"obner Basis method.

Let us point out that our final proofs are done by hand computations, without using computer algebra.

\begin{theo}\label{HUN}
The system~\eqref{C_hn} with arbitrary even $n\geq2$ and $a=2b,\;c=-b,\; b\neq 0$, has an isochronous center at the origin with non zero Urabe function

$$h(X)=\frac{bX^{n-1}}{\sqrt{1+b^2X^{2n-2}}}$$

\end{theo}
\begin{proof}
When $a=2b\;{\rm{ and }}\;c=-b$ the system ~\eqref{C_hn} becomes
\begin{equation}\label{HO}
 \left. \begin{array}{rl} \dot x&= - y+ 2bx^{n-1}y\\
 \dot y&= x + bx^{n-2}y^2 -bx^n \end{array}\right\}. 
\end{equation}
The change of variables 
$(x,y)\longmapsto(x/b,y/b)$reduces the system~\eqref{HO} to the form
$$\left. \begin{array}{rl} \dot x&= - y+ 2x^{n-1}y\\
 \dot y&= x + x^{n-2}y^2 -x^n \end{array}\right\}$$ 
which is reducible to the Liénard type equation~\eqref{L2} with
$$f(x)={\frac {\left( -1+2\,n \right){x}^{n-2}}{1-2\,{x}^{n-1}}}\;\; {\rm{and}}\;\;g(x)=\left( 1-2\,{x}^{n-1} \right)  \left( x-{x}^{n} \right)$$
Then $$F(x)=\int_0^x f(s) ds ={\frac { 1-2\,n} {2n-2}\ln  \left( 1-2\,{x}^{n-1} \right) }$$ which gives the right hand side of the equality ~\eqref{CRI}
$$g(x){e^{ F(x)}}=\frac{x(1-x^{n-1})}{(1-2\,x^{n-1})^{\frac{1}{2n-2}}}.$$
On the other hand, $e^{2\,F(x)}=(1-2\,x^{n-1})^{\frac{1-2\,n}{n-1}}$.

>From the equation~\eqref{xi} we compute $$X(x)=\sqrt{2\int_0^{x} g(s)e^{2F(s)}ds}=\frac{x}{(1-2\,x^{n-1})^{\frac{1}{2n-2}}}$$
and $$h(X(x))=\frac{{X(x)}^{n-1}}{\sqrt{1+{X(x)}^{2n-2}}}=\frac{x^{n-1}}{1-x^{n-1}}$$
Then we compute the left hand side of the equality~\eqref{CRI} :
$$\frac{X(x)}{1+h(X(x))}= \frac{\frac{x}{(1-2\,x^{n-1})^{\frac{1}{2n-2}}}}{1+\frac{x^{n-1}}{1-x^{n-1}}}=\frac{x(1-x^{n-1})}{(1-2\,x^{n-1})^{\frac{1}{2n-2}}}$$
Which proves that the equality~\eqref{CRI} is satisfied.
Let us stress that the above computations remain valid for every $n\geq2$. Nevertheless, for $n$ odd $h$ is not an odd function and thus it is not an Urabe function which is odd by definition.\end{proof}
\begin{theo} For arbitrary $n\geq 2$, the system~\eqref{HO} has the following first integral
\begin{equation*}
H_{\eqref{HO}}={\frac {\left({x}^{2}+{y}^{2}\right)^{n-1}}{{2b{x}^{n-1}-1}}}.
\end{equation*}
\end{theo}
\begin{proof}
Using formula~\eqref{L2FI}, one easily computes by {\it{Maple}} the first integral for $n=4,6,8$. The obtained results strongly suggest the veracity of the formula for $H_{\eqref{HO}}$.
Now one easily can check by hand that $H_{\eqref{HO}}$ is a first integral.
\end{proof}
Let us return to system~\eqref{C_hn}.
It is well known that for $n=2$, this system has an isochronous center in exactly four cases, so called {\it Loud isochronous centers} (see \cite{L+,C3+}). They correspond to ($a=b,\,c=0$), ($a=\frac{b}{2},\,c=-\frac{b}{4}$), ($a=2b,\,c=-b$) and ($b=\frac{a}{4},\,c=0$).
The first two are those from Theorem~\ref{TU0} , the third is the one from Theorem~\ref{HUN}. 

Let us note the Taylor expansion of the Urabe function $h(X)=c_1X+c_3X^3+\ldots$.  As noted at the begining of the Section, for $n=4$ one has exactly $3$ cases of isochronous centers. Why such a difference? The difference is in the algebraic structure of the equations generated by C-algorithm. For $n=2$, the second of such equations is $-3\,c_{{1}}+a-2\,c-b=0$  and $c_{{1}}$ can be non zero, while for $n\geq3$, the second such equation is always $c_{{1}}=0$. Thus the freedom for existence of non zero Urabe function is greater for $n=2$ than for $n\geq3$. 

 
 \section{Non-homogeneous perturbations of degree four with zero Urabe function}
 Taking into account the condition $g'(x)+f(x)g(x)=1$ from Corollary~A, using {\it{Maple}} one easily obtains the following Theorem.
\begin{theo} 
The system~\eqref{C_4} has an isochronous center at the origin $O$ with zero Urabe function only in one of the following six cases, where one supposes that all denominators are non zero polynomials.

\fbox{I}
\begin{equation*}
 \left. \begin{array}{rl} \dot x&= -y+b_{{02}}xy+a_{{21}}{x}^{2}y+a_{{31}}{x}^{3}y               \\
 \dot y&= x+b_{{02}}{y}^{2}+a_{{21}}x{y}^{2}+a_{{31}}{x}^{2}{y}^{2}                         \end{array}\right\} 
\end{equation*}

\medskip

\fbox{II}
\begin{equation*}
 \left. \begin{array}{rl} \dot x&= -y+b_{{02}}xy-\frac{3}{2}\,b_{{30}}{x}^{2}y+b_{{02}}b_{{30}}{x}^{3
}y       \\
 \dot y&=x+b_{{02}}{y}^{2}+b_{{30}}{x}^{3}-\frac{9}{2}\,b_{{30}}x{y}^{2}+3\,b_
{{02}}b_{{30}}{x}^{2}{y}^{2}
                          \end{array}\right\} 
\end{equation*}

\medskip

\fbox{III}
\begin{equation*}
 \left. \begin{array}{rl} \dot x&= -y+ \left( b_{{02}}+2\,b_{{20}} \right) xy+a_{{21}}{x}^{2}y+
 \left( b_{{20}}a_{{21}}-b_{{02}}{b_{{20}}}^{2}-4\,{b_{{
20}}}^{3} \right) {x}^{3}y \\
&~\\
 \dot y&=  x+b_{{20}}{x}^{2}+b_{{02}}{y}^{2}+ \left( a_{{21}}+b_{{02}
}b_{{20}}+4\,{b_{{20}}}^{2} \right) x{y}^{2}\\
&+ \left( 2\,b_{{20
}}a_{{21}}-2\,b_{{02}}{b_{{20}}}^{2}-8\,{b_{{20}}}^{3}
 \right) {x}^{2}{y}^{2}
                        \end{array}\right\} 
\end{equation*}

\medskip

\fbox{IV}
\begin{equation*}
 \left. \begin{array}{rl} \dot x&= -y+{\frac { \left( -9\,b_{{30}}{b_{{20}}}^{2}-{b_{{20}}}^{2}a
_{{21}}+2\,{b_{{20}}}^{4}+4\,b_{{30}}a_{{21}}+6\,{b_{{30
}}}^{2} \right)}{b_{{20}} \left( -{b_{{20}}}^{2}+4\,b_{{30
}} \right) }} xy\\
&~\\
&+a_{{21}}{x}^{2}y+{\frac {b_{{30}} \left( -2\,b_{{30}}{b_{{20}}}^{2}-{b_{{20}}}^{2}a_{{21}}+4\,b_{{30}}a_{
{21}}+6\,{b_{{30}}}^{2} \right) }{b_{{20}} \left( -{b
_{{20}}}^{2}+4\,b_{{30}} \right) }}{x}^{3}y
                                     \\
                                     &~\\
 \dot y&= x+b_{{20}}{x}^{2}+{\frac { \left( -17\,b_{{30}}{b_{{20}}}^{2}
-{b_{{20}}}^{2}a_{{21}}+4\,{b_{{20}}}^{4}+4\,b_{{30}}a_{{
21}}+6\,{b_{{30}}}^{2} \right) }{b_{{20}} \left( -{b_{{
20}}}^{2}+4\,b_{{30}} \right) }}{y}^{2}\\
&~\\
&+b_{{30}}{x}^{3}+2\,{\frac {
 \left( -{b_{{20}}}^{2}a_{{21}}+4\,b_{{30}}a_{{21}}+b_{{
30}}{b_{{20}}}^{2}-3\,{b_{{30}}}^{2} \right)}{-{b_{{
20}}}^{2}+4\,b_{{30}}}} x{y}^{2}\\
&~\\
&+3\,{\frac {b_{{30}} \left( -2\,b_{{3,0
}}{b_{{20}}}^{2}-{b_{{20}}}^{2}a_{{21}}+4\,b_{{30}}a_{{2
1}}+6\,{b_{{30}}}^{2} \right) }{b_{{20}} \left( -
{b_{{20}}}^{2}+4\,b_{{30}} \right) }}{x}^{2}{y}^{2}
                         \end{array}\right\} 
\end{equation*}

\newpage

\fbox{V}
\begin{equation*}
 \left. \begin{array}{rl} \dot x&= -y-{\frac { \left( -32\,b_{{40}}b_{{30}}{b_{{20}}}^{2}+42\,{b
_{{40}}}^{2}b_{{20}}+8\,b_{{40}}{b_{{20}}}^{4}+b_{{40}}
{b_{{30}}}^{2}-2\,{b_{{30}}}^{2}{b_{{20}}}^{3}+7\,{b_{{30}
}}^{3}b_{{20}} \right) }{-{b_{{30}}}^{2}{b_{{20}}}^{2}+4\,a
_{{2,4,0}}{b_{{20}}}^{3}-18\,b_{{40}}b_{{30}}b_{{20}}+27\,
{b_{{40}}}^{2}+4\,{b_{{30}}}^{3}}}xy\\
&~\\
&-{\frac { \left( 6\,{b_{{30
}}}^{4}-2\,{b_{{20}}}^{2}{b_{{30}}}^{3}-27\,b_{{40}}b_{{20
}}{b_{{30}}}^{2}+8\,b_{{30}}b_{{40}}{b_{{20}}}^{3}+39\,b_{
{30}}{b_{{40}}}^{2}-4\,{b_{{40}}}^{2}{b_{{20}}}^{2}
 \right) }{-{b_{{30}}}^{2}{b_{{20}}}^{2}+4\,b_{{40}}{
b_{{20}}}^{3}-18\,b_{{40}}b_{{30}}b_{{20}}+27\,{b_{{40}
}}^{2}+4\,{b_{{30}}}^{3}}}{x}^{2}y\\
&~\\
&-2\,{\frac {b_{{40}} \left( -{b_{{30
}}}^{2}{b_{{20}}}^{2}-14\,b_{{40}}b_{{30}}b_{{20}}+18\,{b_
{{40}}}^{2}+4\,b_{{40}}{b_{{20}}}^{3}+3\,{b_{{30}}}^{3}
 \right) }{-{b_{{30}}}^{2}{b_{{20}}}^{2}+4\,b_{{40}}{
b_{{20}}}^{3}-18\,b_{{40}}b_{{30}}b_{{20}}+27\,{b_{{40}
}}^{2}+4\,{b_{{30}}}^{3}}} {x}^{3}y                                    \\
&~\\
 \dot y&= x+b_{{20}}{x}^{2}-{\frac { \left( -68\,b_{{40}}b_{{30}}{b_{{2
0}}}^{2}+96\,{b_{{40}}}^{2}b_{{20}}+16\,b_{{40}}{b_{{20
}}}^{4}+b_{{40}}{b_{{30}}}^{2}-4\,{b_{{30}}}^{2}{b_{{20}}}
^{3}+15\,{b_{{30}}}^{3}b_{{20}} \right) }{-{b_{{30}}}^
{2}{b_{{20}}}^{2}+4\,b_{{40}}{b_{{20}}}^{3}-18\,b_{{40}}b_
{{30}}b_{{20}}+27\,{b_{{40}}}^{2}+4\,{b_{{30}}}^{3}}}{y}^{2}\\
&~\\
&+b_{{
30}}{x}^{3}-2\,{\frac { \left( 9\,{b_{{30}}}^{4}-3\,{b_{{20}}
}^{2}{b_{{30}}}^{3}-40\,b_{{40}}b_{{20}}{b_{{30}}}^{2}+12
\,b_{{30}}b_{{40}}{b_{{20}}}^{3}+60\,b_{{30}}{b_{{40}}}
^{2}-8\,{b_{{40}}}^{2}{b_{{20}}}^{2} \right) }{-{b_{{30}}}^{2}{b_{{20}}}^{2}+4\,b_{{40}}{b_{{20}}}^{3}-18\,b_{{40}}b_{{30}}b_{{20}}+27\,{b_{{40}}}^{2}+4\,{b_{{30}}}^{3}}
}x{y}^{2}\\
&~\\
&-8\,{\frac {b_{{40}} \left( -{b_{{30}}}^{2}{b_{{20}}}^{2}-14
\,b_{{40}}b_{{30}}b_{{20}}+18\,{b_{{40}}}^{2}+4\,b_{{40
}}{b_{{20}}}^{3}+3\,{b_{{30}}}^{3} \right) }{-{b_{
{30}}}^{2}{b_{{20}}}^{2}+4\,b_{{40}}{b_{{20}}}^{3}-18\,b_{
{40}}b_{{30}}b_{{20}}+27\,{b_{{40}}}^{2}+4\,{b_{{30}}}^
{3}}}{x}^{2}{y}^{2}+b_{{40}}{x}^{4}
                         \end{array}\right\} 
\end{equation*}

\medskip

\fbox{VI}
\begin{equation*}
 \left. \begin{array}{rl} \dot x&={-y+ \left( 2\,b_{{20}}+b_{{02}} \right) xy-\frac{{b_{{20}}}^{2}}{32} \left( 567\,{Z}^{2}b_{{20}}+24
\,Zb_{{02}}-804\,Zb_{{20}}+113\,b_{{20}}-8\,b_{{02}} \right) {x}^{3}y
} 
\\
&~\\
&+{\frac {b_{{20}}
 \left( 13032\,{Z}^{2}b_{{02}}-4488\,Zb_{{02}}+68719\,{Z}^{2}b_{{20}}-
22970\,Zb_{{20}}+384\,b_{{2}}+1943\,b_{{20}} \right) }{24(387\,{Z
}^{2}-144\,Z+13)}}{x}^{2}y\\
&~\\
 \dot y&= x+b_{{20}}{x}^{2}+b_{{02}}{y}^{2}-{\frac {Z{b_{{20}}}^{2} \left( 
27\,Z-17 \right) }{4+12\,Z}}{x}^{3} \\
&-\frac{b_{{20}}}{32} \left( -108\,{Z}^{2}
b_{{02}}+1053\,{Z}^{2}b_{{20}}+152\,Zb_{{02}}-1964\,Zb_{{20}}-76\,b_{{02}
}+155\,b_{{20}} \right) x{y}^{2}\\
&-\frac{{b_{{20}}}^{2}}{8} \left( 567\,{Z}^{
2}b_{{20}}+24\,Zb_{{02}}-804\,Zb_{{20}}+113\,b_{{20}}-8\,b_{{02}}
 \right) {x}^{2}{y}^{2}+1/4\,Z{b_{{20}}}^{3}{x}^{4}
 \end{array}\right\} 
\end{equation*}
where  $Z$ is the only real  root of the equation $27\,{{\it s}}^{3}-47\,{{\it 
s}}^{2}+13\,{\it s}-1=0$, which is equal to $$Z={\frac {1}{81}}\,\sqrt [3]{39428+324\,\sqrt {93}}+{\frac {1156}{81}}\,
{\frac {1}{\sqrt [3]{39428+324\,\sqrt {93}}}}+{\frac {47}{81}}$$
\end{theo} 
  

\section{Non-homogeneous  perturbations of degree four}

Let us consider  system~\eqref{C_4}. We would like to identify all its isochronous centers by C-algorithm, without taking into account the nature of its Urabe function.
 In full generality, this problem cannot be attained by our actual computer possiblities.
 Indeed, we do not succeed to compute a Gr\"obner basis for the nine C-algorithm generated polynomials on $9$ unknown $\{a_{ij}\}$ and $\{b_{sr}\}$ of all even degrees between $2$ and $18$.
  
   Inspecting the system under consideration one sees that the annulation of some parameters $\{a_{ij}\}$ and $\{b_{rs}\}$ will substantially simplify the system. This is the reason of our choice of two families presented below.

\subsection{First family}

Let us assume $a_{11}=b_{30}=0$, in this case system~\eqref{C_4} reduces to the system

\begin{equation}\label{C4B}\left.\begin{array}{rl} \dot x& = - y + a_{21}x^2y +a_{31}x^3y\\
 \dot y &= x + b_{20}x^2 + b_{02}y^2  + b_{12}xy^2+b_{22}x^2y^2+b_{40}x^4 \end{array}\right\} \end{equation}

\begin{theo}\label{FAM1} 
 The system~\eqref{C4B}  has an isochronous center at $O$ if and only if its parameters satisfy one of the folowing 6 conditions :

\fbox{I} 
\begin{equation*}
 \left. \begin{array}{rl} \dot x&= y+{\frac {b_{22}}{4}}{x}^{3}y \\
 \dot y&=x+b_{22}{x}^{2}{y}^{2}-{\frac {3b_{22} }{16}} {x}^{4}
\end{array}\right\} 
\end{equation*}

\medskip

\fbox{II}
\begin{equation*}
 \left. \begin{array}{rl} \dot x&= - y+ 2b_{22}x^{3}y\\
 \dot y&= x + b_{22}x^{2}y^2 -b_{22}x^4 \end{array}\right\}. 
\end{equation*}

\medskip

\fbox{III}
\begin{equation*}
 \left. \begin{array}{rl} \dot x&= -y+a_{{21}}{x}^{2}y+a_{{31}}{x}^{3}y                \\
                          \dot y&= x+a_{{21}}x{y}^{2}+a_{{31}}{x}^{2}{y}^{2}                                          \end{array}\right\} 
\end{equation*}

\medskip

\fbox{IV}
\begin{equation*}
 \left. \begin{array}{rl} \dot x&=-y+\frac{2}{3}\,{b_{{20}}}^{2}{x}^{2}y-\frac{4}{3}\,{b_{{20}}}^{3}{x}^{3}y                 \\
                          \dot y&=x+b_{{20}}{x}^{2}-2\,b_{{20}}{y}^{2}+\frac{8}{3}\,{b_{{20}}}^{2}x{y}^{2}-\frac{8}{3}
\,{b_{{20}}}^{3}{x}^{2}{y}^{2}
                                          \end{array}\right\} 
\end{equation*}

\medskip

\fbox{V}

\begin{equation*}
 \left. \begin{array}{rl} \dot x&=-y+\frac{2}{3}\,{b_{{20}}}^{2}{x}^{2}y-\frac{4}{3}\,{b_{{20}}}^{3}{x}^{3}y                 \\
                          \dot y&=x+b_{{20}}{x}^{2}-2\,b_{{20}}{y}^{2}+\frac{8}{3}\,{b_{{20}}}^{2}x{y}^{2}-\frac{8}{3}
\,{b_{{20}}}^{3}{x}^{2}{y}^{2}
                                           \end{array}\right\} 
\end{equation*}

\medskip

\fbox{VI}

\begin{equation*}
 \left. \begin{array}{rl} \dot x&= -y+{b_{{20}}}^{2} \left( 2+{\frac {a_{{31}}}{{b_{{20}}}^{3}}}
 \right) {x}^{2}y+a_{{31}}{x}^{3}y
                \\
                          \dot y&= x+b_{{20}}{x}^{2}-2\,b_{{20}}{y}^{2}+{b_{{20}}}^{2} \left( 4+{
\frac {a_{{31}}}{{b_{{20}}}^{3}}} \right) x{y}^{2}+2\,a_{{31}}{x}^{
2}{y}^{2}
                                          \end{array}\right\} 
                                         \end{equation*}
 
 where  $b_{20}\neq0$.

\end{theo}


 \begin{proof}
 C-algorithm gives the six candidates to be isochronous centers. We had to derive 19 times to get
  the necessary conditions of isochronicity. 
 
To apply succesfully the C-algorithm, we use the two tricks explained in Appendix, Sec.7.2 : homogenization and reduction of the dimension of the parameters space by one.
This leads to the proof that the cases I-VI of Theorem \ref{FAM1} satisfy the necessary conditions of isochronicity.
We check that the necessary conditions are also sufficient by direct application of Corollary~A  to the cases I, III-VI.
Indeed, in all those four cases $g'(x)+f(x)g(x)=1.$
For sufficiently small $x$ the case II is a particular case of the system \eqref{HO} when $n=4$, studied in Theorem~\ref{HUN}.
\end{proof}

Let us note that among the above six cases only the cases I, II and III with $a_{21}=0$ represent the homogeneous perturbations.
All other cases are non-homogeneous.

Note also that the above three homogeneous cases were already identified by Theorems \ref{HUN} and \ref{TU0}. But contrary to the quoted Theorems, here we have the exhaustive list of isochronous centers for $n=4$.

\subsection{Second family}

Consider system~\eqref{C_4}, with $a_{11}=a_{21}=0$.
We obtain the seven parameter real system of degree $4$.
\begin{equation}\label{C4B1}
\left. \begin{split} \dot x &= - y +a_{31}x^3y\\
 \dot y &= x + b_{20}x^2 + b_{02}y^2 + b_{30}x^3  + b_{12}xy^2+b_{22}x^2y^2+b_{40}x^4 \end{split}\right\} 
 \end{equation}

\begin{theo}
  The system~\eqref{C4B1}  has an isochronous center at $O$ if and only if its parameters satisfy one of the folowing seven cases :
  
 The three cases I, II and III with $a_{21}=0$ come from Theorem~\ref{FAM1} and correspond to homogeneous perturbations. The following four cases  correspond to the non-homogeneous perturbations.

\fbox{IV}

\begin{equation*}
 \left. \begin{array}{rl} \dot x&=-y+\frac{1}{4}\,{b_{{02}}}^{3}{x}^{3}y                 \\
                          \dot y&=x-\frac{1}{2}\,b_{{02}}{x}^{2}+b_{{02}}{y}^{2}+\frac{1}{2}\,{b_{{02}}}^{2}x{y}^{2}+
\frac{1}{2}\,{b_{{02}}}^{3}{x}^{2}{y}^{2}
                                           \end{array}\right\} 
\end{equation*}

\medskip

\fbox{V}
 \begin{equation*}
 \left. \begin{array}{rl} \dot x&=-y+{\frac {1}{192}}\,{b_{{02}}}^{3} \left( -21+5\,\sqrt {33} \right) 
{x}^{3}y
                \\
                          \dot y&=x- \frac{1}{2}\,b_{{02}}{x}^{2}+b_{{02}}{y}^{2}+\frac{1}{48}\,{b_{{02}}}^{2} \left( 
9-\sqrt {33} \right) {x}^{3}\\
&+\frac{1}{16}\,{b_{{02}}}^{2} \left( -1+\sqrt {33
} \right) x{y}^{2}+{\frac {1}{64}}\,{b_{{02}}}^{3} \left( -21+5\,
\sqrt {33} \right) {x}^{2}{y}^{2}
                                           \end{array}\right\} 
\end{equation*}

\medskip

\fbox{VI} 
  \begin{equation*}
 \left. \begin{array}{rl} \dot x&= -y-{\frac {1}{192}}\,{b_{{02}}}^{3} \left( 21+5\,\sqrt {33} \right) {
x}^{3}y
                \\
                          \dot y&=x-\frac{1}{2}\,b_{{02}}{x}^{2}+b_{{02}}{y}^{2}+\frac{1}{48}\,{b_{{02}}}^{2} \left( 
9+\sqrt {33} \right) {x}^{3}\\
&-\frac{1}{16}\,{b_{{02}}}^{2} \left( 1+\sqrt {33}
 \right) x{y}^{2}-{\frac {1}{64}}\,{b_{{02}}}^{3} \left( 21+5\,\sqrt 
{33} \right) {x}^{2}{y}^{2}
                                          \end{array}\right\} 
\end{equation*}

\medskip

\fbox{VII} 
\begin{equation*}
 \left. \begin{array}{rl} \dot x&= -y+{\frac {2}{3549}}\,{\frac {{b_{{20}}}^{3} \left( -43\,{t}^{2/3}-
7670\,\sqrt {3297}+12112\,\sqrt [3]{t}+52\,\sqrt [3]{t}\sqrt {3297}-
336886 \right) }{{t}^{2/3}}}{x}^{3}y

                \\
&~\\
                          \dot y&=x-2\,b_{{2,0}}{y}^{2}-{\frac {1}{10647}}\,{\frac {b_{{2,0}} \left( -
3822\,b_{{2,0}}{t}^{2/3}-6242964\,b_{{2,0}}-127764\,b_{{2,0}}\sqrt {
3297}+159432\,b_{{2,0}}\sqrt [3]{t} \right) }{{t}^{2/3}}}x{y}^{2}\\
&~\\
&-{
\frac {1}{10647}}\,{\frac {b_{{2,0}} \left( 1032\,{b_{{2,0}}}^{2}{t}^{
2/3}+184080\,{b_{{2,0}}}^{2}\sqrt {3297}-290688\,{b_{{2,0}}}^{2}\sqrt 
[3]{t}-1248\,{b_{{2,0}}}^{2}\sqrt {3297}\sqrt [3]{t}+8085264\,{b_{{2,0
}}}^{2} \right)}{{t}^{2/3}}}{x}^{2}{y}^{2}\\
&~\\
&+b_{{2,0}}{x}^{2}-{\frac {1
}{10647}}\,{\frac {b_{{2,0}} \left( -53144\,b_{{2,0}}\sqrt [3]{t}+
2080988\,b_{{2,0}}+42588\,b_{{2,0}}\sqrt {3297}-5824\,b_{{2,0}}{t}^{2/
3} \right)}{{t}^{2/3}}} {x}^{3}\\
&~\\
&-{\frac {1}{10647}}\,{\frac {b_{{2,0}}
 \left( -2150\,{b_{{2,0}}}^{2}{t}^{2/3}-11926\,{b_{{2,0}}}^{2}\sqrt [3
]{t}+234\,{b_{{2,0}}}^{2}\sqrt {3297}\sqrt [3]{t}+1085248\,{b_{{2,0}}}
^{2}+18720\,{b_{{2,0}}}^{2}\sqrt {3297} \right)}{{t}^{2/3}}} {x}^{4}
 \end{array}\right\} 
\end{equation*}
where $t=22868+468\,\sqrt {3297}$

\end{theo}


\begin{proof}
Thanks to C-algorithm we obtain the necessary conditions for the isochronicity of the center at the origin for system~\eqref{C4B1} and we establish the seven cases given in the theorem.

We check that the obtained necessary conditions are also sufficient by direct application of Corollary~A  to the cases IV-VII.
Indeed, in all those four cases $g'(x)+f(x)g(x)=1$ for sufficiently small $x$.

\end{proof}
The centers $I-III$ of Theorem 2 have been already identified in \cite{ROMA4}.


\section{Non-homogeneous perturbations of degree five with zero Urabe function}
By Corollary~A the problem is reduced to solving the equation  $g'(x)+f(x)g(x)=1$, with $f$ and $g$ defined in Sec.1.1 with respect to the system~\eqref{C_5}. In this case the equation $g'(x)+f(x)g(x)=1$ is equivalent to some system of $8$ polynomials depending on $12$ unknown $\{a_{ij}\}$ and $\{b_{sr}\}$ of degree $2, 2, 2, 1, 2, 2, 2, 2$. Applying the Gr\"obner basis method one obtains a basis of $90$ polynomials whose degrees varies between 1 and 8. This system is too hard to handle. By inspecting the system one sees that when $b_{50}=0$ it is very much simplified. This is the reason of our choice $b_{50}=0$.
\begin{theo} 
The system~\eqref{C_5}  where $b_{50}=0$ has an isochronous center at the origin $O$ with zero Urabe function only in one of the following 8 cases, where one supposes that all denominators are non zero polynomials.

\fbox{I}

\begin{equation*}
 \left. \begin{array}{rl} \dot x&= -y+a_{{11}}xy+b_{{12}}{x}^{2}y+a_{{31}}{x}^{3}y+b_{{32}}{x}^{4}y
\\
 \dot y&=x+a_{{11}}{y}^{2}+b_{{12}}x{y}^{2}+a_{{31}}{x}^{2}{y}^{2}+b_{{32}}{x}^
{3}{y}^{2}
   \end{array}\right\} 
\end{equation*}

\medskip

\fbox{II}
\begin{equation*}
 \left. \begin{array}{rl} \dot x&= -y+a_{{11}}xy+a_{{31}}{x}^{3}y-\frac{3}{4}\,a_{{11}}a_{{31}}{x}^{4}y\\
 \dot y&=x+a_{{11}}{y}^{2}+4\,a_{{31}}{x}^{2}{y}^{2}-\frac{3}{4}\,a_{{31}}{x}^{4}-3\,a_
{{11}}a_{{31}}{x}^{3}{y}^{2}
 \end{array}\right\} 
\end{equation*}

\medskip

\fbox{III}
\begin{equation*}
 \left. \begin{array}{rl} \dot x&=-y+{\frac {a_{{31}}}{b_{{30}}}}xy+ \left( 3\,b_{{30}}+b_{{12}}
 \right) {x}^{2}y+a_{{31}}{x}^{3}y+ \left( \frac{9}{2}\,{b_{{30}}}^{2}+b_{{12}
}b_{{30}} \right) {x}^{4}y

              \\
 \dot y&=x+{\frac {a_{{31}}}{b_{{30}}}}{y}^{2}+b_{{30}}{x}^{3}+b_{{12}}x{y}^{2}
+3\,a_{{31}}{x}^{2}{y}^{2}+ \left( {\frac {27}{2}}\,{b_{{30}}}^{2}+3\,
b_{{12}}b_{{30}} \right) {x}^{3}{y}^{2}
     \end{array}\right\} 
\end{equation*}

\medskip

\fbox{IV}
\begin{equation*}
 \left. \begin{array}{rl} \dot x&=-y+a_{{11}}xy+ \left( b_{{12}}-2\,{b_{{20}}}^{2}-a_{{11}}b_{{20}}
 \right) {x}^{2}y+a_{{31}}{x}^{3}y+\\
& \left( -b_{{12}}{b_{{20}}}^{2}+4\,
{b_{{20}}}^{4}+2\,a_{{11}}{b_{{20}}}^{3}+a_{{31}}b_{{20}} \right) {x}^
{4}y

             \\
&~\\
 \dot y&= x+b_{{20}}{x}^{2}+ \left( -2\,b_{{20}}+a_{{11}} \right) {y}^{2}+b_{{12
}}x{y}^{2}+ \left( b_{{12}}b_{{20}}-4\,{b_{{20}}}^{3}-2\,a_{{11}}{b_{{
20}}}^{2}+a_{{31}} \right) {x}^{2}{y}^{2}\\
&+ \left( -2\,b_{{12}}{b_{{20}
}}^{2}+8\,{b_{{20}}}^{4}+4\,a_{{11}}{b_{{20}}}^{3}+2\,a_{{31}}b_{{20}}
 \right) {x}^{3}{y}^{2}

    \end{array}\right\} 
\end{equation*}

\medskip

\fbox{V}
\begin{equation*}
 \left. \begin{array}{rl} \dot x&=-y+a_{{11}}xy+a_{{31}}{x}^{3}y\\
&~\\
&-{\frac { \left( 13\,{b_{{30}}}^{2}b_{{20}}-11\,b_{{30}}{
b_{{20}}}^{3}-5\,b_{{30}}a_{{11}}{b_{{20}}}^{2}+2\,{b_{{20}}}^{5}+a_{{
11}}{b_{{20}}}^{4}+a_{{31}}{b_{{20}}}^{2}-4\,a_{{31}}b_{{30}}+4\,a_{{
11}}{b_{{30}}}^{2} \right) }{b_{{20}} \left( 4\,b_{{30}}-{b_{{
20}}}^{2} \right) }}{x}^{2}y\\
&~\\
&-{\frac {b_{{30}} \left( -b_{{30}
}a_{{11}}{b_{{20}}}^{2}+7\,{b_{{30}}}^{2}b_{{20}}-2\,b_{{30}}{b_{{20}}
}^{3}-4\,a_{{31}}b_{{30}}+4\,a_{{11}}{b_{{30}}}^{2}+a_{{31}}{b_{{20}}}
^{2} \right) }{b_{{20}} \left( 4\,b_{{30}}-{b_{{20}}}^{2}
 \right) }}{x}^{4}y
 \\
&~\\
 \dot y&=x+b_{{20}}{x}^{2}+ \left( -2\,b_{{20}}+a_{{11}} \right) {y}^{2}+b_{{30
}}{x}^{3}\\
&~\\
&-{\frac { \left( 25\,{b_{{30}}}^{2}b_{{20}}-22\,b_{{30}}{b_{{
20}}}^{3}-9\,b_{{30}}a_{{11}}{b_{{20}}}^{2}+4\,{b_{{20}}}^{5}+2\,a_{{
11}}{b_{{20}}}^{4}+a_{{31}}{b_{{20}}}^{2}-4\,a_{{31}}b_{{30}}+4\,a_{{
11}}{b_{{30}}}^{2} \right) }{b_{{20}} \left( 4\,b_{{30}}-{b_{{
20}}}^{2} \right) }}x{y}^{2}\\
&~\\
&+{\frac { \left( 7\,{b_{{30}}}^{2}b_{{20}}-2\,b_{{
30}}{b_{{20}}}^{3}-b_{{30}}a_{{11}}{b_{{20}}}^{2}-2\,a_{{31}}{b_{{20}}
}^{2}+8\,a_{{31}}b_{{30}}+4\,a_{{11}}{b_{{30}}}^{2} \right) }{4\,b_{{30}}-{b_{{20}}}^{2}}}{x}^{2}{y}
^{2}\\
&~\\
&-3\,{\frac {b_{{30}} \left( -b_{{30}
}a_{{11}}{b_{{20}}}^{2}+7\,{b_{{30}}}^{2}b_{{20}}-2\,b_{{30}}{b_{{20}}
}^{3}-4\,a_{{31}}b_{{30}}+4\,a_{{11}}{b_{{30}}}^{2}+a_{{31}}{b_{{20}}}
^{2} \right) }{b_{{20}} \left( 4\,b_{{30}}-{b_{{20}}}^{2
} \right) }}{x}^{3}{y}^{2}

    \end{array}\right\} 
\end{equation*}

\medskip

\fbox{VI}
\begin{equation*}
 \left. \begin{array}{rl} \dot x&= -y+{\frac { \left( 108\,{b_{{40}}}^{2}-42\,b_{{40}}{b_{{20}}}^{3}+81\,
a_{{31}}b_{{40}}+{b_{{20}}}^{6}-3\,{b_{{20}}}^{3}a_{{31}} \right) }{
{b_{{20}}}^{2} \left( -{b_{{20}}}^{3}+27\,b_{{40}} \right) }}xy\\
&~\\
&+3\,{
\frac { \left( -3\,b_{{40}}{b_{{20}}}^{3}-{b_{{20}}}^{3}a_{{31}}+27\,a
_{{31}}b_{{40}}+36\,{b_{{40}}}^{2} \right) }{b_{{20}} \left( -
{b_{{20}}}^{3}+27\,b_{{40}} \right) }}{x}^{2}y+a_{{31}}{x}^{3}y\\
&~\\
&+3\,{\frac {b_{
{40}} \left( -{b_{{20}}}^{3}a_{{31}}+36\,{b_{{40}}}^{2}+27\,a_{{31}}b_
{{40}} \right) }{{b_{{20}}}^{2} \left( -{b_{{20}}}^{3}+27\,b_{
{40}} \right) }}{x}^{4}y

             \\
&~\\
 \dot y&=x+b_{{20}}{x}^{2}+3\,{\frac { \left( {b_{{20}}}^{6}-32\,b_{{40}}{b_{{
20}}}^{3}+36\,{b_{{40}}}^{2}+27\,a_{{31}}b_{{40}}-{b_{{20}}}^{3}a_{{31
}} \right) }{{b_{{20}}}^{2} \left( -{b_{{20}}}^{3}+27\,b_{{40}}
 \right) }}{y}^{2}\\
&~\\
&+\frac{1}{3}\,{b_{{20}}}^{2}{x}^{3}+6\,{\frac { \left( -4\,b_{{40}}
{b_{{20}}}^{3}-{b_{{20}}}^{3}a_{{31}}+27\,a_{{31}}b_{{40}}+36\,{b_{{40
}}}^{2} \right) }{b_{{20}} \left( -{b_{{20}}}^{3}+27\,b_{{40}}
 \right) }}x{y}^{2}\\
&~\\
&-3\,{\frac { \left( {b_{{20}}}^{3}a_{{31}}-27\,a_{{31}}b_{{
40}}+12\,{b_{{40}}}^{2} \right) }{-{b_{{20}}}^{3}+27\,b_
{{40}}}}{x}^{2}{y}^{2}\\
&~\\
&+b_{{40}}{x}^{4}+12\,{\frac {b_{{40}} \left( -{b_{{20}}}^{3}a_
{{31}}+36\,{b_{{40}}}^{2}+27\,a_{{31}}b_{{40}} \right) }
{{b_{{20}}}^{2} \left( -{b_{{20}}}^{3}+27\,b_{{40}} \right) }}{x}^{3}{y}^{2}

     \end{array}\right\} 
\end{equation*}
\medskip

\fbox{VII}
\begin{equation*}
 \left. \begin{array}{rl} \dot x&=-y+a_{{11}}xy-3\,{\frac {b_{{30}} \left( 13\,{b_{{40}}}^{2}+2\,{b_{{30
}}}^{3} \right) }{27\,{b_{{40}}}^{2}+4\,{b_{{30}}}^{3}}}{x}^{2}y\\
&~\\
&-{
\frac { \left( 5\,{b_{{30}}}^{3}b_{{40}}+36\,{b_{{40}}}^{3}-27\,a_{{11
}}b_{{30}}{b_{{40}}}^{2}-4\,a_{{11}}{b_{{30}}}^{4} \right) }{
27\,{b_{{40}}}^{2}+4\,{b_{{30}}}^{3}}}{x}^{3}y\\
&~\\
&+{\frac {b_{{40}} \left( b_{{40}
}{b_{{30}}}^{2}+27\,a_{{11}}{b_{{40}}}^{2}+4\,a_{{11}}{b_{{30}}}^{3}
 \right) }{27\,{b_{{40}}}^{2}+4\,{b_{{30}}}^{3}}}{x}^{4}y
             \\
&~\\
 \dot y&=x+a_{{11}}{y}^{2}+b_{{30}}{x}^{3}-6\,{\frac {b_{{30}} \left( 3\,{b_{{
30}}}^{3}+20\,{b_{{40}}}^{2} \right) }{27\,{b_{{40}}}^{2}+4\,{
b_{{30}}}^{3}}}x{y}^{2}\\
&~\\
&-3\,{\frac { \left( -27\,a_{{11}}b_{{30}}{b_{{40}}}^{2}
-4\,a_{{11}}{b_{{30}}}^{4}+7\,{b_{{30}}}^{3}b_{{40}}+48\,{b_{{40}}}^{3
} \right) }{27\,{b_{{40}}}^{2}+4\,{b_{{30}}}^{3}}}{x}^{2}{y}^{2}\\
&~\\
&+b_{{
40}}{x}^{4}+4\,{\frac {b_{{40}} \left( b_{{40}}{b_{{30}}}^{2}+27\,a_{{
11}}{b_{{40}}}^{2}+4\,a_{{11}}{b_{{30}}}^{3} \right) }{
27\,{b_{{40}}}^{2}+4\,{b_{{30}}}^{3}}}{x}^{3}{y}^{2}

     \end{array}\right\} 
\end{equation*}

\medskip

\fbox{VIII}
$$
 \left. \begin{split}\dot x&= - y+ a_{11}xy+ a_{21}x^2y + a_{31}x^3y+ a_{41}x^4y\\
 \dot y&= x + b_{20}x^2 + b_{30}x^3+ b_{02}y^2 + b_{12}xy^2 + b_{22}x^2y^2+ b_{32}x^3y^2+ b_{40}x^4+ b_{50}x^5 \end{split}\right\} $$
where
$$\;a_{{11}}=\frac{P_{11}}{Q},\;a_{{21}}=\frac{b_{20}P_{21}}{Q},a_{{31}}=\frac{P_{31}}{Q},\;a_{{41}}=\frac{P_{41}}{Q},
\;{b}_{{02}}=\frac{P_{02}}{Q},\;b_{{22}}=\frac{P_{22}}{Q},\;{b}_{{32}}=\frac{P_{32}}{Q}$$
 and
\begin{equation*}
 \left. \begin{array}{rl}
 
P_{{11}}&=-41\,b_{{40}}{b_{{30}}}^{2}b_{{20}}+2\,b_{{40}}b_{{12}}{b_{{
20}}}^{3}-9\,b_{{40}}b_{{30}}b_{{12}}b_{{20}}+9\,{b_{{30}}}^{4}-8\,b_{
{40}}{b_{{20}}}^{5}\\
&-50\,{b_{{40}}}^{2}{b_{{20}}}^{2}+{\frac {27}{2}}\,
b_{{12}}{b_{{40}}}^{2}+60\,{b_{{40}}}^{2}b_{{30}}-10\,{b_{{30}}}^{3}{b
_{{20}}}^{2}+2\,{b_{{30}}}^{2}{b_{{20}}}^{4}\\
&-\frac{1}{2}\,{b_{{30}}}^{2}{b_{{
20}}}^{2}b_{{12}}+44\,b_{{40}}b_{{30}}{b_{{20}}}^{3}+2\,b_{{12}}{b_{{
30}}}^{3}\\
&~\\ 
 
P_{{21}}&=-13\,b_{{40}}{b_{{30}}}^{2}b_{{20}}+2\,b_{{40}}b_{{12}}{b_{{
20}}}^{3}-9\,b_{{40}}b_{{30}}b_{{12}}b_{{20}}+3\,{b_{{30}}}^{4}-4\,{b_
{{40}}}^{2}{b_{{20}}}^{2}\\
&+{\frac {27}{2}}\,b_{{12}}{b_{{40}}}^{2}+21\,
{b_{{40}}}^{2}b_{{30}}-{b_{{30}}}^{3}{b_{{20}}}^{2}-\frac{1}{2}\,{b_{{30}}}^{2
}{b_{{20}}}^{2}b_{{12}}+4\,b_{{40}}b_{{30}}{b_{{20}}}^{3}+2\,b_{{12}}{
b_{{30}}}^{3}\\
&~\\ 
 
P_{{31}}&=9\,{b_{{30}}}^{5}+2\,{b_{{30}}}^{4}b_{{12}}-3\,{b_{{30}}}^{4}
{b_{{20}}}^{2}-46\,b_{{40}}b_{{20}}{b_{{30}}}^{3}-\frac{1}{2}\,{b_{{30}}}^{3}b
_{{12}}{b_{{20}}}^{2}\\

&+14\,{b_{{30}}}^{2}{b_{{20}}}^{3}b_{{40}}-9\,{b_{
{30}}}^{2}b_{{40}}b_{{12}}b_{{20}}+60\,{b_{{30}}}^{2}{b_{{40}}}^{2}+2
\,b_{{30}}b_{{40}}b_{{12}}{b_{{20}}}^{3}\\
&+{\frac {27}{2}}\,b_{{30}}b_{{
12}}{b_{{40}}}^{2}+20\,b_{{30}}{b_{{20}}}^{2}{b_{{40}}}^{2}-36\,{b_{{
40}}}^{3}b_{{20}}-8\,{b_{{20}}}^{4}{b_{{40}}}^{2}\\
&~\\

P_{{41}}&={\frac {27}{2}}\,b_{{12}}{b_{{40}}}^{3}+60\,{b_{{40}}}^{3}b_{
{30}}+2\,b_{{40}}b_{{12}}{b_{{30}}}^{3}+12\,{b_{{40}}}^{2}b_{{30}}{b_{
{20}}}^{3}-8\,{b_{{40}}}^{3}{b_{{20}}}^{2}\\
&-\frac{1}{2}\,b_{{40}}{b_{{30}}}^{2}
{b_{{20}}}^{2}b_{{12}}-9\,{b_{{40}}}^{2}b_{{30}}b_{{12}}b_{{20}}+9\,b_
{{40}}{b_{{30}}}^{4}+2\,{b_{{40}}}^{2}b_{{12}}{b_{{20}}}^{3}-3\,b_{{40
}}{b_{{30}}}^{3}{b_{{20}}}^{2}\\
&-40\,{b_{{40}}}^{2}{b_{{30}}}^{2}b_{{20}}\\
&~\\

P_{{02}}&=-16\,b_{{40}}{b_{{20}}}^{5}+80\,b_{{40}}b_{{30}}{b_{{20}}}^{3
}-104\,{b_{{40}}}^{2}{b_{{20}}}^{2}+4\,{b_{{30}}}^{2}{b_{{20}}}^{4}-18
\,{b_{{30}}}^{3}{b_{{20}}}^{2}\\
&-41\,b_{{40}}{b_{{30}}}^{2}b_{{20}}+2\,b
_{{40}}b_{{12}}{b_{{20}}}^{3}-9\,b_{{40}}b_{{30}}b_{{12}}b_{{20}}+9\,{
b_{{30}}}^{4}+{\frac {27}{2}}\,b_{{12}}{b_{{40}}}^{2}\\
&+60\,{b_{{40}}}^{2}b_{{30}}-\frac{1}{2}\,{b_{{30}}}^{2}{b_{{20}}}^{2}b_{{12}}+2\,b_{{12}}{b_{{
30}}}^{3}\\
&~\\

P_{{22}}&=180\,{b_{{30}}}^{2}{b_{{40}}}^{2}-144\,{b_{{40}}}^{3}b_{{20}}
-32\,{b_{{20}}}^{4}{b_{{40}}}^{2}+6\,b_{{30}}b_{{40}}b_{{12}}{b_{{20}}
}^{3}+27\,{b_{{30}}}^{5}+6\,{b_{{30}}}^{4}b_{{12}}\\
&-9\,{b_{{30}}}^{4}{b
_{{20}}}^{2}+{\frac {81}{2}}\,b_{{30}}b_{{12}}{b_{{40}}}^{2}-144\,b_{{
40}}b_{{20}}{b_{{30}}}^{3}+44\,{b_{{30}}}^{2}{b_{{20}}}^{3}b_{{40}}\\
&+88
\,b_{{30}}{b_{{20}}}^{2}{b_{{40}}}^{2}-27\,{b_{{30}}}^{2}b_{{40}}b_{{
12}}b_{{20}}-\frac{3}{2}\,{b_{{30}}}^{3}b_{{12}}{b_{{20}}}^{2}
\\
&~\\

P_{{32}}&=54\,b_{{12}}{b_{{40}}}^{3}+240\,{b_{{40}}}^{3}b_{{30}}+8\,b_{
{40}}b_{{12}}{b_{{30}}}^{3}+48\,{b_{{40}}}^{2}b_{{30}}{b_{{20}}}^{3}-
32\,{b_{{40}}}^{3}{b_{{20}}}^{2}\\
&-2\,b_{{40}}{b_{{30}}}^{2}{b_{{20}}}^{
2}b_{{12}}-36\,{b_{{40}}}^{2}b_{{30}}b_{{12}}b_{{20}}+36\,b_{{40}}{b_{
{30}}}^{4}+8\,{b_{{40}}}^{2}b_{{12}}{b_{{20}}}^{3}\\
&-12\,b_{{40}}{b_{{30
}}}^{3}{b_{{20}}}^{2}-160\,{b_{{40}}}^{2}{b_{{30}}}^{2}b_{{20}}\\
&~\\

Q&=4\,b_{{40}}{b_{{20}}}^{4}-18\,b_{{40}}b_{{30}}{b_{{20}}}^{2}+27\,{b_
{{40}}}^{2}b_{{20}}-{b_{{30}}}^{2}{b_{{20}}}^{3}+4\,b_{{20}}{b_{{30}}}
^{3}

\end{array}\right.
\end{equation*}
\end{theo}


\section{Abel polynomial system}

By planar Abel system of order $n$ we mean the system
\begin{equation}\label{Ab}
\left. \begin{split}
\dot x&=-y\\ \displaystyle\dot y&=\displaystyle\sum_{k=0}^{n} P_k (x){y}^k
\end{split}\right\}
\end{equation}  
where  $\{P_k(x)\}_{0\leq k\leq n}$ are smooth functions.

 This section is concerned by the following Abel system

\begin{equation}\label{I}
\left.\begin{split} 
\dot x&=-y\\
 \dot y&=x(1+P(y)),
      \end{split}\right\} \end{equation}
with $ P(y)=a_1 y + a_2 y^2+a_3 y^3+....+a_n y^n$, $a_k \in \R$, for $k =0,\ldots, n$.
This is a particular Abel system~\eqref{Ab} where $P_k(x):=a_kx$, ${0\leq k\leq n}$.
 \subsection{Characterization of isochronous centers}
 
System~\eqref{I} is reducible (see Sec.1.1, Case 2) to the Liénard type equation \eqref{L2} with $f(x)$ and $g(x)$
defined by \eqref{GLb}.
 Definitions~\eqref{xixi}, \eqref{xi} and Theorems B and C from Sec.1.1 remain valid. Applied to the Abel system~\eqref{I} they give :

\begin{theo}\label{ABEL} The origin $O$ is a center for the system~\eqref{I}.

The center at $O$, is isochronous if and only if  there exists an odd function $h$ defined in some neighborhood of $0\in \R$ 
which satisfies the following conditions
\begin{equation*}
\frac {X}{1+h(X)} = x,
\end{equation*} 

\begin{equation*}\phi(x)= X(x) + \int_0^{X(x)} h(t) dt\end{equation*} 
and  $X(x)\phi(x) > 0$ for $x\neq 0$.\\
 
In particular, when $P$ is an even polynomial then the origin $O$ is an isochronous center if and only if $P=0$.
 \end{theo}
 
\begin{proof}
$xg(x)=x^2(1+P(x))>0$ for $x\neq0$ and $|x|$ small enougth. Then Theorem B implies that the origin $O$ is a center of the system~\eqref{I}.

Now
     \begin{equation*}
F(x) = \int_0^x f(s) ds=-\ln(1+P(x)),\end{equation*}
thus
 \begin{equation}\label{P0}
  \phi(x)=\int_0^x e^{F(s)} ds=\int_0^x\frac{ds}{1+P(s)} \end{equation}
   Then we obtain
   \begin{equation*}  g(x) e^{F(x)}=x(1+P(x))e^{-\ln(1+P(x))}=x\end{equation*}  
  Following Theorem~C, one obtains
   $$ \frac {X(x)}{1+h(X(x))}=x$$
  as well as the identity \eqref{bb}.
    
   For the particular case where $P$ is even, it is easy to see that $f$ and $g$ are odd.
   Theorem~C thus implies $h=0$ and consequently $X(x)=x$. From \eqref{bb} one deduces that $\phi(x)=X(x)$. Then \eqref{P0} implies that $P\equiv0$.
   \end{proof}

The following paragraph is devoted to illustrate the last theorem by example.

\subsection{An application}
  Let  us consider the Abel system~\eqref{I} with $n=9$ :
\begin{equation}\label{AB9}
  \left. \begin{split}
\dot x&=-y\\ \dot y&=x+ \sum_{i=1}^9a_i x y^i
    \end{split}\right\}\end{equation}
with  $ a_k \in\mathbb R$, $1\leq k \leq 9$.
As follows from Theorem~\ref{ABEL},  the origin $O$ is always a center for~\eqref{AB9}.

\begin{theo}\label{ISL}
The system \eqref{AB9} has an isochronous center at the origin $0$ only in the case 

 \begin{equation}\label{I3}
\left. \begin{split}
\dot x&=-y\\ \dot y&=x+ a xy+ \frac{a^2}{3} x{y}^2+ \frac{a^3}{27} x{y}^3
    \end{split}\right\} \end{equation} 
    where $a\in \R$
  
\end{theo}
\begin{proof}
We apply C-algorithm  for
\begin{equation*}
f(x)=-\frac{P'(x)}{1+P(x)}\, \hbox{and}\, g(x)=x(1+P(x)),
     \end{equation*} 
where $P(x)=\sum_{i=1}^9a_i x^i$.
We obtain the unique one-parameter family~\eqref{I3}, and computations give the Urabe function 
    \begin{equation*} h(X) = -\frac{a_1X}{3} =\frac{k_1 X}{\sqrt{{k_2}^2+k_3 X^2}}
    \end{equation*}
  with $k_1=-a_1/3,\; k_2=1,\; k_3=0$.
  
\end{proof}
 By the evident rescalling $\frac{a}{3}x\mapsto x$ and $\frac{a}{3}y\mapsto y$ system~\eqref{I3} takes the form
which is a particular case of system \eqref{I}:
 \begin{equation}\label{Ab3}\left.\begin{split}
\dot x&=-y\\ \dot y &=x(1+y)^3.
    \end{split}\right\}  \end{equation}
    The isochronous center at the origin $O$ for system~\eqref{Ab3} was already depicted in \cite{VI+} by showing that system~\eqref{Ab3} commutes with some transversal polynomial system, but neither its first integral nor the linearizing change of coordinates were provided. We shall now compute both of them.%


\begin{itemize}
\item {\underline{\bf{First integral}}}
In the variables $u=y$ and $v=x(1+y)^3$ the system \eqref{Ab3} is reducible to the Liénard type equation $\ddot u+f(u){\dot u}+g(u)=0$ where 
$$f(u)=-\frac{3}{1+u}  ~and  ~g(u)=u(1+u)^3.$$ 
By formula \eqref{L2FI} from Theorem~B one easily obtains that 
$I(u,v)=\frac{u^2}{(1+u)^2}+\frac{v^2}{(1+u)^6}$ is a first integral of the corresponding planar system
\begin{equation*}\left.\begin{split}
\dot u&=v\\ \dot v &=-g(u)-f(u)v^2
    \end{split}\right\}  \end{equation*}
    
    Returning to the variables $(x,y)$ one recovers the first integral of the system~\eqref{Ab3} :
    \begin{equation*}
I_{\eqref{Ab3}}(x,y)=x^2+\frac{y^2}{(1+y)^2}.
\end{equation*}
\item {\underline{\bf{Linearization}}}
For this purpose we use the method of~\cite{GARC} based on the exitence of vector field $Y$ transversal to the vector field defined by the system \eqref{Ab3} and commuting with it.
{\it{Maple}} computations give such a field $Y$:
\begin{equation}\label{ComAb3}\left.\begin{split}
\dot x&=x+xy\\ \dot y &=-{x}^{2}{y}^{3}+{y}^{2}-3\,{x}^{2}{y}^{2}+y-3\,{x}^{2}y-{x}^{2}
    \end{split}\right\}  \end{equation}
    
    Following the method described in \cite{GARC}, we first establish an inverse integrating factor $V(x,y)$ of the system~\eqref{Ab3}:

    $$V(x,y)=- \left( y+1 \right)  \left( {x}^{2}+2\,{x}^{2}y+{x}^{2}{y}^{2}+{y}^{2} \right) $$
 which leads to the first integrals $H$ (already known) and $I$ of the systems \eqref{Ab3} and \eqref{ComAb3} respectively:
 \begin{equation*}
  \left.\begin{split}
H(x,y)&=x^2+\frac{y^2}{(1+y)^2}\\ I(x,y)&=-x+\arctan \left( {\frac { \left( y+1 \right) x}{y}} \right)
    \end{split}\right\}  \end{equation*}
Let us define $\tilde{f}(z)=z$ and $\tilde{g}(z)=tan(z)$.

 By the Theorem~4 of~\cite{GARC} we obtain the linearizing change of coordinates 
 \begin{equation*}
  \left.\begin{split}
u&=\frac{\sqrt{\tilde{f}(H(x,y))}\tilde{g}(I(x,y))}{\sqrt{1+\tilde{g}^2(I(x,y))}}\\
&~\\
 v&=\frac{\sqrt{\tilde{f}(H(x,y))}}{\sqrt{1+\tilde{g}^2(I(x,y))}}
    \end{split}\right\}  \end{equation*}
   {\it{Maple}} produces the following more explicit formulas that we, as usual, reproduce without any change to avoid the misprints:
$$ \left.\begin{split}
u(x,y)&=\frac{-\sqrt {{\frac {{x}^{2}+2\,y{x}^{2}+{y}^{2}{x}^{2}+{y}^{2}}{ \left( y+
1 \right) ^{2}}}}\tan \left( x-\arctan \left( {\frac { \left( y+1
 \right) x}{y}} \right)  \right)} {{\sqrt {1+ \left( \tan
 \left( x-\arctan \left( {\frac { \left( y+1 \right) x}{y}} \right) 
 \right)  \right) ^{2}}}}\\ 
&~\\
v(x,y)&=\sqrt {{\frac {{x}^{2}+2\,y{x}^{2}+{y}^{2}{x}^{2}+{y}^{2}}{ \left( y+1
 \right) ^{2}}}}{\frac {1}{\sqrt {1+ \left( \tan \left( x-\arctan
 \left( {\frac { \left( y+1 \right) x}{y}} \right)  \right)  \right) ^
{2}}}}
.
\end{split}\right\}$$  
    The fact that this change of variables actually is a linearizing one can easily be verified by {\it{Maple}} which gives
$\dot u=-v,\; \dot v=u$ as expected.
\end{itemize}
  
Under the light of Theorem~\ref{ISL},  it is natural to ask if the system~\eqref{I3} is the unique system with isochronous center at the origin 0  in  family~\eqref{I}. Even for 
$n=10$, our actual computer possibilies are not sufficient to give an answer.
  

\section{Appendix}
{
\subsection{C-Algorithm }
Theorem~C (see Sec.1.1) leads to an algorithm, first introduced  in \cite{C3+} (see also\cite{CRC+}), hereafter called C-algorithm, 
which gives necessary conditions for isochronicity
of the center at the origin $O$  for equation~\eqref{L2}.

Below we recall basic steps of the algorithm. 

Let $h$ be the Urabe function defined in the Theorem~C,  and $u ={\phi (x)}$.  The function $\phi$ is invertible around 0 .
\begin{equation}
\label{tg}
	\tilde{g}(u) :=\frac{X}{1+h(X)},
\end{equation}
where now $X$ is considered as a function of $u$. Our further assumption is that  functions $f(x)$ and $g(x)$ depend polynomially on certain parameters $\alpha:=(\alpha_1,\ldots,\alpha_p)\in\R^p$.

By Theorem~C, if the system~\eqref{L2} has isochronous center at the origin $O$, then the Urabe function $h$  must be odd, so we have 
$$h(X)=\sum_{k=0}^{\infty}{c_{2k+1}} X^{2k+1}$$
and moreover, 
\begin{equation}
\label{tgg}
\tilde{g}(u) = g(x)e^{F(x)}, \quad\text{where}\quad x=\phi^{-1}(u).	
\end{equation}
Hence, the right hand sides of~\eqref{tg} and~\eqref{tgg} must be equal. We expand both right hand sides into the Taylor series around 0 and equate the corresponding coefficients. To this end we need to calculate $k$-th derivatives of~\eqref{tg} and~\eqref{tgg}.
  
For~\eqref{tg}, by straightforward differentiation, we have 
$$\frac{d^k\tilde{g}(u)}{du^k} =\frac{d}{dX}\left(\frac{d^{k-1}\tilde{g}(u)}{du^{k-1}} \right) \frac{dX}{du}$$
Using induction, one can show that for~\eqref{tgg} one obtains 
$$\frac{d^k\tilde{g}(u)}{du^k} = e^{(1-k)F(x)}S_k(x),$$
where $S_k(x)$ is a function of $f(x), g(x)$ and their derivatives.

Therefore to compute the first $m$ conditions for isochronicity of system~\eqref{L2} we proceed as follows.
\begin{enumerate}
	\item We fix $m$ and write 
	\[
	h(X) =\sum_{k=1}^{m}{c_{2k-1}} X^{2k-1}+O(X^{2m}), \quad c:=(c_1, c_3, \ldots, c_{2m-1}).
\]
\item Next, we compute 

	\[
	v_k :=\frac{d^k\tilde{g}}{du^k}(0), \quad w_k= S_k(0)
\]
 for $k = 1,\ldots, 2m + 1$. Note that those quantities are polynomials in $\alpha$ and $c$.
 \item By Theorem~C we obtain the equations $v_k = w_k$ for $k = 1,\ldots, 2m + 1$.  Let us note that always $v_1=w_1=1$ and thus the first equation is meaningless.
 
 It appears that we always can eliminate parameters $c$ from these equations. 
 For every $k\geq 0$, $c_{2k+1} $ occours for the first time, and in a linear way, in the equation $v_{2k+2}=w_{2k+2}$. This leads to the formula $c_{2k+1}=\varphi_{2k+1}(\alpha)$ for some multivariate polynomial $\varphi_{2k+1}$. This brings us in a natural way to the consecutive elimination of $c_1, c_3, \ldots, c_{2m-1}$. Finally, we obtain at most $m$ polynomial equations $s_1 = s_2 = s_3 = \ldots= s_M = 0$ with $p$ unknowns $\alpha_i$. These equations denoted $Sys(m)$ give $M$ necessary conditions for isochronicity of system~\eqref{L2}; $ Sys(m)\subset Sys(m+1).$
\end{enumerate} For more details see \cite{C3+,CRC+}.

The progammation of the C-algorithm can be done in different ways. Some of them appear more successful than others. For the purposes of the present paper we used the programm from \cite{BB}.
\subsection{Homogeneization and reduction}
The Gr\"obner bases method for solving the systems of polynomial equations is particularly efficient when all these polynomials are homogeneous or weighted-homogeneous (see \cite{BW+}, Sec.10.2).

A long experience with $C$-algorithm indicates that the following facts are always verified although they are not proved.

Let us consider the system \eqref{GL} written explicitly as
 \begin{equation}\label{C_nn} \qquad \left. \begin{array}{rl} \dot x&= - y+ a_{11}xy+\ldots+ a_{n-1,1}x^{n-1}y\\
 \dot y&= x + b_{20}x^2 +b_{02}y^2 +\ldots+ b_{n-2,2}x^{n-2}y^2+ b_{n,0}x^n\end{array}\right\} 
\end{equation}
which as explained in Introduction, is reducible to the equation~\eqref{L2} 

When considering the particular cases of system \eqref{C_nn} corresponding to homogeneous (resp. non-homogeneous) perturbations of system \eqref{LINC}, the polynomials from $Sys(m)$ are homogeneous (resp. non-homogeneous).\\ 

We note that for $n=3$, C-algorithm succeeds in establishing isochronicity criteria, however for $n=4$ the obtained polynomials from the algorithm are much more involved.
For instance, system~\eqref{C_nn} with $n=4$ reduces to the system \eqref{C_4}.
Now, the first two non-zero  polynomials from $Sys(9)$ are
\begin{equation}
\label{p3}
P_2=3\,a_{{21}}-3\,b_{{12}}+{a_{{11}}}^{2}-b_{{20}}a_{{11}}-9\,b_{{30
}}+4\,{b_{{02}}}^{2}-5\,a_{{11}}b_{{02}}+10\,{b_{{20}}}^{2}+10\,b_
{{20}}b_{{02}}, 
\end{equation}
\begin{equation*}
\begin{split}
P_3&=72\,{a_{{21}}}^{2}+396\,b_{{20}}a_{{11}}b_{{12}}+90\,a_{{11}}b_{{0
2}}b_{{12}}+36\,a_{{11}}b_{{22}}+324\,a_{{31}}b_{{02}}\\
  & -36\,a_{{2
1}}b_{{12}}-468\,b_{{20}}a_{{11}}a_{{21}}+612\,b_{{20}}a_{{21}}
b_{{02}}-4116\,a_{{11}}{b_{{20}}}^{2}b_{{02}}\\
& +108\,b_{{20}}a_{{31}}
 -540\,b_{{30}}a_{{21}}-324\,b_{{40}}a_{{11}}+1566\,b_{{30}}a_{
{11}}b_{{02}}-288\,b_{{20}}b_{{22}}\\
&-459\,b_{{30}}{a_{{11}}}^{2}-
1296\,b_{{40}}b_{{02}}-306\,a_{{21}}a_{{11}}b_{{02}}+1428\,b_{{20
}}{a_{{11}}}^{2}b_{{02}}\\
&+153\,a_{{21}}{a_{{11}}}^{2}-117\,{a_{{11
}}}^{2}b_{{12}}-191\,b_{{20}}{a_{{11}}}^{3}+180
\,b_{{20}}b_{{02}}b_{{12}}
+43\,{a_{{11}}}^{4}\\
&-2319\,b_{{20}}a_{{11}}{b_{{02}}}^{2}-
289\,{a_{{11}}}^{3}b_{{02}}-360\,b_{{02}}b_{{22}}-36\,{b_{{12}}}^
{2}-171\,a_{{21}}{b_{{02}}}^{2}\\
&+513\,b_{{30}}{b_{{02}}}^{2}+537\,{
a_{{11}}}^{2}{b_{{02}}}^{2}+351\,{b_{{02}}}^{2}b_{{12}}-271\,a_{{1
1}}{b_{{02}}}^{3}+542\,b_{{20}}{b_{{02}}}^{3}\\
&+756\,b_{{20}}b_{{30}}b_{{02}}+2268\,b_{{20}}b_{{30}}a_{{11}}-20\,{b_{{02}}}^{4}+1120
\,{b_{{20}}}^{4}+798\,{a_{{11}}}^{2}{b_{{20}}}^{2}\\
&-2240\,a_{{11}}{
b_{{20}}}^{3}-1512\,b_{{20}}b_{{40}}+1008\,{b_{{20}}}^{2}a_{{21}}
-252\,{b_{{20}}}^{2}b_{{12}}\\
&+1806\,{b_{{20}}}^{2}{b_{{02}}}^{2}+2240\,{b_{{20}}}^{3}b_{{02}}
\end{split}
\end{equation*}
To solve the systems of polynomials equations we use Gr\"obner bases. Solving system of 9 (non-zero) equations of $Sys(9)$ requires higher performance computers and ours are not up to it.

 A careful analysis of polynomials from $Sys(m)$  shows that for any $m$ they always are weighted-homogeneous. 
For example, the polynomial $P_2$ given by~\eqref{p3} is weighted-homogeneous if we give weight 2 for $a_{21}$, $b_{12}$ and $b_{30}$, and weight 1 for the remaining variables. 

We observe that all polynomials from $Sys(m)$ are weighted-homogeneous if we choose the following weights 
\begin{enumerate}
	\item $i+j-1$ for parameters $a_{ij}$ and $b_{ij}$
	\item $2i+1$ for $c_{2i+1}$.
\end{enumerate}
We introduce new parameters $A_{ij}$, $B_{ij}$, and $C_{2i+1}$ putting 
\begin{equation}
\label{ABC}
	A_{ij}^{i+j-1}=a_{ij}\quad B_{ij}^{i+j-1}=b_{ij}, \quad C_{2i+1}^{2i+1}=c_{2i+1}
\end{equation}
After this reparametrization system \eqref{C_nn} reads
\begin{equation}\label{C_nh}
 \left. \begin{split} \dot x&= - y+ A_{11}xy+\ldots+ A_{n-1,1}^{n-1}x^{n-1}y\\
 \dot y&= x + B_{20}x^2 +B_{02}y^2 +\ldots+ B_{n-2,2}^{n-1}x^{n-2}y^2+ B_{n,0}^{n-1}x^n\end{split}\right\} 
\end{equation}
As in the case of isochronous center the Urabe function is odd, we search it under the form 
$$h(X)=\sum_{k=0}^{\infty}C_{2k+1}^{2k+1} X^{2k+1} = C_1 X + C_3^{3} X^3 +C_5^{5}X^5 +C_7^{7} X^7 +\ldots$$ 
By a simply use of~\eqref{ABC}, from the isochronicity conditions for system \eqref{C_nh}, expressed in terms of its parameters 
$\left\{A_{ij}\right\}$ and $\left\{B_{rs}\right\}$, it is easy to recover the parameters values 
$\left\{a_{ij}\right\}$ and $\left\{b_{rs}\right\}$ when system \eqref{C_nn} admits isochronous centers at the origin $O$.
 
The described reparametrization gives rise to homogeneous equations and reduces  the number of parameters appearing in~\eqref{C_nh} by one. First we assume  ${B_{20}}=0$, and we solve the isochronicity problem for system~\eqref{C_nh} under this assumption.
Next, for ${B_{20}}\neq0$, we apply to system~\eqref{C_nh} the following change of coordinates
$$
 (x,y) \mapsto(\frac{x}{ B_{20}},\frac{y}{ B_{20}})
$$
Where
$$
\left. 
\begin{array}{rl} \dot x &= - y +\left(\frac{A_{11}}{B_{20}}\right)xy +..+ \left(\frac{B_{n-1,1}}{B_{20}}\right)^{n-1}x^{n-1}y\\
&~\\
 \dot y &= x + x^2 +\left( \frac{B_{02}}{B_{20}}\right)y^2  + ..+\left(\frac{B_{n-2,2}}{B_{2,0}}\right)^{n-1}x^{n-2}y^2+\left(\frac{B_{n,0}}{B_{20}}\right)^{n-1}x^n
 \end{array}\right\}
$$
Hence, without loss of generality  we can put ${B_{20}}=1$, and find the parameters values for which the center is  isochronous. \\

Note that the resolution of the polynomial system issued from the $19$ derivations and associated eliminations for  system~\eqref{C_4} (with 9 parameters), exceed our computer facilities.

\subsection{On efficiency of C-algorithm}
Usually for the search of isochronous centers one uses the method of normal form (see for exemple~\cite{ROSH,ROMA4}). Thus it is interesting to compare C-algorithm with it, when there are applied to systems reducible to Liénard type equations.

As an example for such investigations we choose the Abel systems~\eqref{I} with $2\leq n\leq 9$. The normal form (NF) algorithm used is the one described in~\cite{YU} which is universal and efficient.
As explained in point (3) of Sec.7.1 for the system~\eqref{I} with $n$ parameters one computes 
coefficients $c_1, c_3, \ldots, c_{2m-1}$
of the Urabe function when C-algorithm is used. When the normal form method is used one computes its first $2n+1$ terms.

The results are presented in the table where the time unit is one second.
\newpage

\begin{table}[ht]
\caption{CPU time on Pentium 1,46 GHz}
\begin{center}
\begin{tabular}{|c|c|c|}
\hline\hline
$n$ & C-algorithm & NF algorithm\\
\hline\hline
2 & $\sim$ 0 & 0,060 \\
\hline\hline
3 & 0,001 & 0,160 \\
\hline\hline
4 & 0,004 & 0,784 \\
\hline\hline
5 & 0,008 & 4,728 \\
\hline\hline
6 & 0,016 & 31,430 \\
\hline\hline
7 & 0,052 & 263,033 \\
\hline\hline
8 & 0,116 & 2335,962 \\
\hline\hline
9 & 0,284 & ~\\
\hline
\end{tabular}
\end{center}
\end{table}

\noindent The superiority of C-algorithm is obvious.
}


 \bigskip
 \noindent{\large{{\bf{Acknowledgments}}}} 
\medskip

 We warmly thank Professors  Magali Bardet (University of Rouen, France), Isaac A. Garcia (University of Lleida, Spain) and Andrzej J. Maciejewski (University of Zielona G\`ora, Poland) for helpful discussions and critical remarks. Last but not least, we sincerly thank Marie-Claude Werquin (University Paris 13, France) who corrected and improved our ``scientific English''.

\end{document}